\documentclass{article}
\usepackage{amssymb}
\usepackage{amsmath}
\usepackage{amsthm}
\usepackage{centernot}
\usepackage{color}

\newtheorem{theorem}{Theorem}[section]
\newtheorem{lemma}[theorem]{Lemma}
\newtheorem{proposition}[theorem]{Proposition}
\newtheorem{corollary}[theorem]{Corollary}

\theoremstyle{definition}

\newtheorem{example}[theorem]{Example}
\newtheorem{algorithm}[theorem]{Algorithm}

\theoremstyle{remark}
\newtheorem{remark}[theorem]{Remark}

\numberwithin{equation}{section}

\begin{document}

\title{Frobenius restricted varieties in numerical semigroups}

\author{Aureliano M. Robles-P\'erez\thanks{Both authors are supported by the project MTM2014-55367-P, which is funded by Mi\-nis\-terio de Econom\'{\i}a y Competitividad and Fondo Europeo de Desarrollo Regional FEDER, and by the Junta de Andaluc\'{\i}a Grant Number FQM-343. The second author is also partially supported by Junta de Andaluc\'{\i}a/Feder Grant Number FQM-5849.} \thanks{Departamento de Matem\'atica Aplicada, Universidad de Granada, 18071-Granada, Spain. \newline E-mail: \textbf{arobles@ugr.es}}
	\mbox{ and} Jos\'e Carlos Rosales$^*$\thanks{Departamento de \'Algebra, Universidad de Granada, 18071-Granada, Spain. \newline E-mail: \textbf{jrosales@ugr.es}} }

\date{ }

\maketitle

\begin{abstract}
	The common behaviour of many families of numerical semigroups led up to defining, firstly, the Frobenius varieties and, secondly, the (Frobenius) pseudo-varieties. However, some interesting families are still out of these definitions. To overcome this situation, here we introduce the concept of Frobenius restricted variety (or $R$-variety). We will generalize most of the results for varieties and pseudo-varieties to $R$-varieties. In particular, we will study the tree structure that arise within them.
\end{abstract}

\medskip

\noindent{\it 2010 AMS Classification:} 20M14

\section{Introduction}
In \cite{variedades}, the concept of (Frobenius) variety was introduced in order to unify several results which have appeared in \cite{patterns}, \cite{systems}, \cite{arf}, and \cite{saturated}. Moreover, the work made in \cite{variedades} has allowed to study other notables families of numerical semigroups, such as those that appear in \cite{frases}, \cite{benefits}, \cite{digitales}, and \cite{brazaletes}.

There exist families of numerical semigroups which are not varieties but have a similar structure. For example, the family of numerical semigroups with maximal embedding dimension and fixed multiplicity (see \cite{med}). The study of this family, in \cite{bagsvo}, led to the concept of $m$-variety.

In order to generalize the concepts of variety and $m$-variety, in \cite{pseudovar} were introduced  the (Frobenius) pseudo-varieties. Moreover, recently, the results obtained in \cite{pseudovar} allowed us to study several interesting families of numerical semigroups (for instance, see \cite{incentivos}).

In this work, our aim will be to introduce and study the concept of $R$-variety (that is, Frobenius restricted variety). We will see how it generalizes the concept of pseudo-variety and we will show that there exist significant families of numerical semigroups which are $R$-varieties but not pseudo-varieties.

Let ${\mathbb N}$ be the set of nonnegative integers. A \emph{numerical semigroup} is a subset $S$ of ${\mathbb N}$ such that it is closed under addition, contains the zero element, and ${\mathbb N} \setminus S$ is finite.

It is well known (see \cite[Lemma~4.5]{springer}) that, if $S$ and $T$ are numerical semigroups such that $S \varsubsetneq T$, then $S \cup \{\max(T\setminus S)\}$ is another numerical semigroup. We will denote by ${\mathrm F}_T(S)=\max(T\setminus S)$ and we will call it as the \emph{Frobenius number of $S$ restricted to $T$}.

An \emph{$R$-variety} is a non-empty family  ${\mathcal R}$ of numerical semigroups that fulfills the following conditions.
\begin{enumerate}
	\item ${\mathcal R}$ has a maximum element with respect to the inclusion order (that we will denote by $\Delta({\mathcal R})$).
	\item If $S,T \in {\mathcal R}$, then $S\cap T\in {\mathcal R}$.
	\item If $ S\in {\mathcal R}$ and $S\neq \Delta({\mathcal R})$, then $S\cup \{{\mathrm F}_{\Delta({\mathcal R})}(S)\} \in {\mathcal R}$.
\end{enumerate}

In Section~\ref{allvarieties} we will see that every pseudo-variety is an $R$-variety. Moreover, we will show that, if ${\mathcal V}$ is a variety and $T$ is a numerical semigroup, then ${\mathcal V}_T=\{S \cap T \mid S \in {\mathcal V} \}$ is an $R$-variety. In fact, we will prove that every $R$-variety is of this form.

Let ${\mathcal R}$ be an $R$-variety and let $M$ be a submonoid of $({\mathbb N},+)$. We will say that $M$ is an \emph{${\mathcal R}$-monoid} if it can be expressed as intersection of elements of ${\mathcal R}$. It is clear that the intersection of ${\mathcal R}$-monoids is another ${\mathcal R}$-monoid and, therefore, we can define the ${\mathcal R}$-monoid generated by a subset of $\Delta({\mathcal R})$. In Section~\ref{r-monoids} we will show that every ${\mathcal R}$-monoid admits a unique minimal ${\mathcal R}$-system of generators. In addition, we will see that, if $M$ is an ${\mathcal R}$-monoid and $x\in M$, then $M\setminus \{x\}$ is another ${\mathcal R}$-monoid if and only if $x$ belongs to the minimal ${\mathcal R}$-system of generators of $M$.

In Section~\ref{tree} we will show that the elements of an $R$-variety, ${\mathcal R}$, can be arranged in a tree with root $\Delta({\mathcal R})$. Moreover, we will prove that the set of children of a vertex $S$, of such a tree, is equal to $\left\{S\setminus\{x\} \mid x \right.$ is an element of the minimal ${\mathcal R}$-system of generators of $S$ and $\left. x>{\mathrm F}_{\Delta({\mathcal R})}(S) \right\}$. This fact will allow us to show an algorithmic process in order to recurrently build the elements of an $R$-variety.

Finally, in Section~\ref{r-var-fns} we will see that, in general and contrary to what happens with varieties and pseudo-varieties, we cannot define the smallest $R$-variety that contains a given family ${\mathcal F}$ of numerical semigroups. Nevertheless, we will show that, if $\Delta$ is a numerical semigroup such that $S\subseteq \Delta$ for all $S\in {\mathcal F}$, then there exists the smallest $R$-variety (denoted by ${\mathcal R}({\mathcal F},\Delta)$) containing ${\mathcal F}$ and having $\Delta$ as maximum (with respect the inclusion order). Moreover, we will prove that ${\mathcal R}({\mathcal F},\Delta)$ is finite if and only if ${\mathcal F}$ is finite. In such a case, that fact will allow us to compute, for a given ${\mathcal R}({\mathcal F},\Delta)$-monoid, its minimal ${\mathcal R}({\mathcal F},\Delta)$-system of generators. In this way, we will obtain an algorithmic process to determine all the elements of ${\mathcal R}({\mathcal F},\Delta)$ by starting from ${\mathcal F}$ and $\Delta$.

Let us observe that the proofs, of some results of this work, are similar to the proofs of the analogous results for varieties and pseudo-varieties. However, in order to get a self-contained paper, we have not omitted several of such proofs.

\section{Varieties, pseudo-varieties, and $R$-varieties}\label{allvarieties}

It is said that $M$ is a \emph{submonoid} of $({\mathbb N},+)$ if $M$ is a subset of ${\mathbb N}$ which is closed for the addition and such that $0\in M$. It particular, if $S$ is a submonoid of $({\mathbb N},+)$ such that ${\mathbb N} \setminus S$ is finite, then $S$ is a numerical semigroup.

Let $A$ be a non-empty subset of ${\mathbb N}$. Then it is denoted by $\langle A \rangle$ the submonoid of $({\mathbb N}, +)$ generated by $A$, that is,
$$\langle A \rangle = \{ \lambda_1 a_1+ \cdots +\lambda_n a_n \mid n \in {\mathbb N} \setminus \{0\}, \; a_1,\ldots,a_n \in A ,\; \lambda_1,\ldots,\lambda_n \in {\mathbb N} \}.$$
It is well known (see for instance \cite[Lemma~2.1]{springer}) that $\langle A \rangle$ is a numerical semigroup if and only if $\gcd(A) =1$.

Let $M$ be a submonoid of $({\mathbb N},+)$ and let $A\subseteq {\mathbb N}$. If $M= \langle A \rangle$, then it is said that $A$ is a \emph{system of generators} of $M$. Moreover, it is said that $A$ is a \emph{minimal system of generators} of $M$ if $M \not= \langle B \rangle$ for all $B \varsubsetneq A$. It is a classical result that every submonoid $M$ of $({\mathbb N},+)$ has a unique minimal system of generators (denoted by $\mathrm{msg}(M)$) which, in addition, is finite (see for instance \cite[Corollary~2.8]{springer}).

Let $S$ be a numerical semigroup. Being that ${\mathbb N} \setminus S$ is finite, it is possible to define several notable invariants of $S$. One of them is the \emph{Frobenius number} of $S$ (denoted by ${\mathrm F}(S)$) which is the greatest integer that does not belong to $S$  (see \cite{alfonsin}). Another one is the \emph{genus} of $S$ (denoted by ${\mathrm g}(S)$) which is the cardinality of ${\mathbb N} \setminus S$.

Let $S$ be a numerical semigroup different from ${\mathbb N}$. Then it is obvious that $S \cup \{{\mathrm F}(S)\}$ is also a numerical semigroup. Moreover, from \cite[Proposition~7.1]{springer}, we have that $T$ is a numerical semigroup with ${\mathrm g}(T)=g+1$ if and only if there exist a numerical semigroup $S$ and $x \in \mathrm{msg}(S)$ such that ${\mathrm g}(S)=g$, $x>{\mathrm F}(S)$, and $T=S \setminus \{x\}$. This result is the key to build the set of all numerical semigroups with genus $g+1$ when we have the set of all numerical semigroups with genus $g$ (see \cite[Proposition~7.4]{springer}).

In \cite{variedades} it was introduced the concept of \emph{(Frobenius) variety} in order to generalize the previous situation to some relevant families of numerical semigroups.

It is said that a non-empty family of numerical semigroups ${\mathcal V}$ is a \emph{(Frobenius) variety} if the following conditions are verified.
\begin{enumerate}
	\item If $S,T \in {\mathcal V}$, then $S\cap T\in {\mathcal V}$.
	\item If $ S\in {\mathcal V}$ and $S\neq {\mathbb N}$, then $S\cup \{{\mathrm F}(S)\} \in {\mathcal V}$.
\end{enumerate}

However, there exist families of numerical semigroups that are not varieties, but have a very similar behavior. By studying these families of numerical semigroups, we introduced in \cite{pseudovar} the concept of \emph{(Frobenius) pseudo-variety}.

It is said that a non-empty family of numerical semigroups ${\mathcal P}$ is a \emph{(Frobenius) pseudo-variety} if the following conditions are verified.
\begin{enumerate}
	\item ${\mathcal P}$ has a maximum element with respect to the inclusion order (that we will denote by $\Delta({\mathcal P})$).
	\item If $S,T \in {\mathcal P}$, then $S\cap T\in {\mathcal P}$.
	\item If $ S\in {\mathcal P}$ and $S\neq \Delta({\mathcal P})$, then $S\cup \{{\mathrm F}(S)\} \in {\mathcal P}$.
\end{enumerate}

From the definitions, it is clear that every variety is a pseudo-variety. Moreover, as a consequence of \cite[Proposition~1]{pseudovar}, we have the next result.

\begin{proposition}\label{prop1}
	Let ${\mathcal P}$ be a pseudo-variety. Then ${\mathcal P}$ is a variety if and only if ${\mathbb N} \in {\mathcal P}$.
\end{proposition}

The following result asserts that the concept of $R$-variety generalizes the concept of pseudo-variety.

\begin{proposition}\label{prop2}
	Every pseudo-variety is an $R$-variety.
\end{proposition}

\begin{proof}
	Let ${\mathcal P}$ be a pseudo-variety. In order to prove that ${\mathcal P}$ is an $R$-variety, we have to show that, if $S\in {\mathcal P}$ and $S\not=\Delta({\mathcal P})$, then $S\cup\{{\mathrm F}_{\Delta({\mathcal P})}(S)\} \in {\mathcal P}$. Since ${\mathcal P}$ is a pseudo-variety, we know that $S\cup \{{\mathrm F}(S)\} \in {\mathcal P}$. Thus, to finish the proof, it is enough to see that ${\mathrm F}(S) = {\mathrm F}_{\Delta({\mathcal P})}(S)$. On the one hand, it is clear that ${\mathrm F}_{\Delta({\mathcal P})}(S) \leq {\mathrm F}(S)$. On the other hand, since $S\cup\{{\mathrm F}(S)\} \in {\mathcal P}$, then we have that ${\mathrm F}(S) \in \Delta({\mathcal P})$. Therefore, ${\mathrm F}(S) \in \Delta({\mathcal P})\setminus S$ and, consequently, ${\mathrm F}(S) \leq {\mathrm F}_{\Delta({\mathcal P})}(S)$.
\end{proof}

In the next example we see that there exist $R$-varieties that are not pseudo-varieties.

\begin{example}\label{exmp3}
	Let ${\mathcal R}$ be the set formed by all numerical semigroups which are contained in the numerical semigroup $\langle 5,7,9 \rangle$. It is clear that ${\mathcal R}$ is an $R$-variety. However, since $S=\langle 5,7,9 \rangle \setminus \{5\} \in {\mathcal R}$, $S\not= \Delta({\mathcal R}) = \langle 5,7,9 \rangle$, ${\mathrm F}(S)=13$, and $S\cup \{13\} \notin {\mathcal R}$, we have that ${\mathcal R}$ is not a pseudo-variety.
\end{example}

Generalizing the above example, we can obtain several $R$-varieties, most of which are not pseudo-varieties.

\begin{enumerate}
	\item Let $T$ be a numerical semigroup. Then ${\mathcal L}_T = \{S \mid S$ is a numerical semigroup and $S \subseteq T\}$ is an $R$-variety. Observe that ${\mathcal L}_T$ is the set formed by all numerical subsemigroups of $T$.
	\item Let $S_1$ and $S_2$ be two numerical semigroups such that $S_1 \subseteq S_2$. Then $[S_1,S_2]=\{S \mid S$ is a numerical semigroup and $S_1 \subseteq S \subseteq S_2\}$ is an $R$-variety.
	\item Let $T$ be a numerical semigroup and let $A \subseteq T$. Then ${\mathcal R}(A,T) = \{ S \mid S$ is a numerical semigroup and $A \subseteq S \subseteq T \}$ is an $R$-variety. Observe that both of the previous examples are particular cases of this one.
\end{enumerate}

\begin{remark}
	Let $p,q$ be relatively prime integers such that $1<p<q$. Let us take the numerical semigroups $S_1=\langle p,q \rangle$ and $S_2=\frac{S_1}{2}=\left\{ s\in{\mathbb N} \mid 2s \in S_1 \right\}$. In \cite{kunz, kunz-waldi}, Kunz and Waldi study the family of numerical semigroups $[S_1,S_2]$, which is an $R$-variety but not a pseudo-variety.
\end{remark}

The next result establishes when an $R$-variety is a pseudo-variety.

\begin{proposition}\label{prop4}
	Let ${\mathcal R}$ be an $R$-variety. Then ${\mathcal R}$ is a pseudo-variety if and only if ${\mathrm F}(S) \in \Delta({\mathcal R})$ for all $S \in {\mathcal R}$ such that $S \not= \Delta({\mathcal R})$.
\end{proposition}

\begin{proof}
	\emph{(Necessity.)} If ${\mathcal R}$ is a pseudo-variety and $S \in {\mathcal R}$ with $S \not= \Delta({\mathcal R})$, then $S \cup \{{\mathrm F}(S)\} \in {\mathcal R}$. Therefore, ${\mathrm F}(S) \in \Delta({\mathcal R})$.
	
	\emph{(Sufficiency.)} In order to show that ${\mathcal R}$ is a pseudo-variety, it will be enough to see that $S \cup \{{\mathrm F}(S)\} \in {\mathcal R}$ for all $S \in {\mathcal R}$ such that $S \not= \Delta({\mathcal R})$. For that, since ${\mathrm F}(S) \in \Delta({\mathcal R})$, then it is clear that ${\mathrm F}_{\Delta({\mathcal R})}(S) = {\mathrm F}(S)$ and, therefore, $S \cup \{{\mathrm F}(S)\} = S \cup \{{\mathrm F}_{\Delta({\mathcal R})}(S)\} \in {\mathcal R}$.
\end{proof}

An immediate consequence of Propositions~\ref{prop1} and \ref{prop4} is the following result.

\begin{corollary}\label{cor5}
	Let ${\mathcal R}$ be an $R$-variety. Then ${\mathcal R}$ is a variety if and only if ${\mathbb N} \in {\mathcal R}$.
\end{corollary}

Our next purpose, in this section, will be to show that to give an $R$-variety is equivalent to give a pair $({\mathcal V},T)$ where ${\mathcal V}$ is a variety and $T$ is a numerical semigroup. Before that we need to introduce some concepts and results.

Let $S$ be a numerical semigroup. Then we define recurrently the following sequence of numerical semigroups.
\begin{itemize}
	\item $S_0=S$,
	\item if $S_i\neq {\mathbb N}$, then $S_{i+1}=S_i\cup\{{\mathrm F}(S_i)\}$.
\end{itemize}

Since ${\mathbb N} \setminus S$ is a finite set with cardinality equal to ${\mathrm g}(S)$, then we get a finite chain of numerical semigroups $S=S_0 \varsubsetneq S_1\varsubsetneq \cdots \varsubsetneq S_{{\mathrm g}(S)}={\mathbb N}$. We will denote by ${\mathrm C}(S)$ the set $\{S_0,S_1,\ldots, S_{{\mathrm g}(S)} \}$ and will say that it is the \emph{chain of numerical semigroups associated to $S$}. If ${\mathcal F}$ is a non-empty family of numerical semigroups, then we will denote by ${\mathrm C}({\mathcal F})$ the set $\bigcup_{S \in {\mathcal F}}{\mathrm C}(S)$.

Let ${\mathcal F}$ be a non-empty family of numerical semigroups. We know that there exists the smallest variety containing ${\mathcal F}$ (see \cite{variedades}). Moreover, by \cite[Thoerem~4]{variedades}, we have the next result.

\begin{proposition}\label{prop6}
	Let ${\mathcal F}$ be a non-empty family of numerical semigroups. Then the smallest variety containing ${\mathcal F}$ is the set formed by all finite intersections of elements of ${\mathrm C}({\mathcal F})$.
\end{proposition}

Now, let ${\mathcal R}$ be an $R$-variety. By applying repeatedly that, if $S \in {\mathcal R}$ and $S \not= \Delta({\mathcal R})$, then $S \cup \{{\mathrm F}_{\Delta({\mathcal R})}(S)\} \in {\mathcal R}$, we get the following result.

\begin{lemma}\label{lem7}
	Let ${\mathcal R}$ be an $R$-variety. If $S \in {\mathcal R}$ and $n \in {\mathbb N}$, then $S \cup \{ x\in \Delta({\mathcal R}) \mid x \geq n \} \in {\mathcal R}$.
\end{lemma}

We are ready to show the announced result.

\begin{theorem}\label{thm8}
	Let ${\mathcal V}$ be a variety and let $T$ be a numerical semigroup. Then ${\mathcal V}_T = \{ S \cap T \mid S \in {\mathcal V}\}$ is an $R$-variety. Moreover, every $R$-variety is of this form.
\end{theorem}

\begin{proof}
	By Proposition~\ref{prop1}, we know that, if ${\mathcal V}$ is a variety, then ${\mathbb N} \in {\mathcal V}$ and, therefore, $T$ is the maximum of ${\mathcal V}_T$ (that is, $T=\Delta({\mathcal V}_T)$). On the other hand, it is clear that, if $S_1,S_2 \in {\mathcal V}_T$, then $S_1 \cap S_2 \in {\mathcal V}_T$.
	
	Now, let $S \in {\mathcal V}$ such that $S \cap T \not= T$ and let us have $t={\mathrm F}_T(S\cap T)$. In order to conclude that ${\mathcal V}_T$ is an $R$-variety, we will see that $(S\cap T) \cup \{t\} \in {\mathcal V}_T$. First, let us observe that $t=\max(T \setminus (S\cap T)) = \max(T\setminus S)$. Then, because $S \in {\mathcal V}$ and ${\mathcal V}$ is a variety, we can easily deduce that $\bar{S} = S \cup \{t,\to \} \in {\mathcal V}$. Moreover, $(S\cap T) \cup \{t\} \subseteq (S\cap T) \cup (\{t,\to \} \cap T) = \bar{S} \cap T$. Let us see now that $\bar{S} \cap T \subseteq (S\cap T) \cup \{t\}$. In other case, there exists $t'>t$ such that $t' \in T$ and $t'\notin S$, in contradiction with the maximality of $t$. Therefore, $(S\cap T) \cup \{t\} = \bar{S} \cap T$ and $\bar{S} \in {\mathcal V}$. Consequently, $(S \cap T)\cup \{t\} \in {\mathcal V}_T$.
	
	Let ${\mathcal R}$ be an $R$-variety and let ${\mathcal V}$ be the smallest variety containing ${\mathcal R}$. To conclude the proof of the theorem, we will see that ${\mathcal R} = {\mathcal V}_{\Delta({\mathcal R})}$. It is clear that ${\mathcal R} \subseteq {\mathcal V}_{\Delta({\mathcal R})}$. Thus, let us see the reverse one. For that, we will prove that, if $S \in {\mathcal V}$, then $S \cap \Delta({\mathcal R}) \in {\mathcal R}$. In effect, by Proposition~\ref{prop6} we have that, if $S \in {\mathcal V}$, then there exist $S_1,\ldots,S_k \in {\mathrm C}({\mathcal R})$ such that $S=S_1\cap \cdots \cap S_k$. Therefore, $S\cap \Delta({\mathcal R})=(S_1\cap \Delta({\mathcal R}))\cap \cdots \cap (S_k\cap \Delta({\mathcal R}))$. Since ${\mathcal R}$ is an $R$-variety, then ${\mathcal R}$ is closed under finite intersections. Thereby, to see that $S \cap \Delta({\mathcal R}) \in {\mathcal R}$, it is enough to show that $S_i \cap \Delta({\mathcal R}) \in {\mathcal R}$ for all $i\in \{1,\ldots,k\}$. Since $S_i \in {\mathrm C}({\mathcal R})$, then it is clear that there exist $S'_i \in {\mathcal R}$ and $n_i \in {\mathbb N}$ such that $S_i=S'_i \cup \{n_i,\to\}$. Therefore, $S_i \cap \Delta({\mathcal R}) = S'_i \cup \{ x\in \Delta({\mathcal R}) \mid x \geq n_i\} \in {\mathcal R}$, by applying Lemma~\ref{lem7}.
\end{proof}

The above theorem allows us to give many examples of $R$-varieties starting from already known varieties.
\begin{enumerate}
	\item Let us observe that, if ${\mathcal V}$ is a variety and $T\in {\mathcal V}$, then ${\mathcal V}_T = \{ S \cap T \mid S \in {\mathcal V}\} = \{ S \in {\mathcal V} \mid S \subseteq T\}$ is an $R$-variety contained in ${\mathcal V}$. Thus, for instance, we have that the set formed by all Arf numerical semigroups, which are contained in a certain Arf numerical semigroup, is an $R$-variety.
	\item Observe also that, if ${\mathcal V}$ is a variety and $T$ is a numerical semigroup such that $T \notin {\mathcal V}$, then ${\mathcal V}_T = \{ S \cap T \mid S \in {\mathcal V}\}$ is an $R$-variety not contained in ${\mathcal V}$ (because $T\in {\mathcal V}_T$ and $T \notin {\mathcal V}$). Let us take, for example, the variety ${\mathcal V}$ of all Arf numerical semigroups and $T=\langle 5, 8 \rangle \notin {\mathcal V}$. In such a case, ${\mathcal V}_T$ is the $R$-variety formed by the numerical semigroups which are the intersection of an Arf numerical semigroup and $T$.
\end{enumerate}

\begin{corollary}\label{cor9}
	Let ${\mathcal R}$ be an $R$-variety and let $U$ be a numerical semigroup. Then ${\mathcal R}_U = \{S \cap U \mid S \in {\mathcal R} \}$ is an $R$-variety.
\end{corollary}

\begin{proof}
	By applying Theorem~\ref{thm8}, we have that there exist a variety ${\mathcal V}$ and a numerical semigroup $T$ such that ${\mathcal R} = {\mathcal V}_T = \{ S \cap T \mid S \in {\mathcal V}\}$. Therefore, ${\mathcal R}_U = \{ S \cap T \cap U \mid S \in {\mathcal V}\} = {\mathcal V}_{T \cap U}$, which is clearly an $R$-variety (by Theorem~\ref{thm8} again). 
\end{proof}
\label{}
The next result says us that Theorem~\ref{thm8} remains true when variety is replaced with pseudo-variety.

\begin{corollary}\label{cor10}
	Let ${\mathcal P}$ be a pseudo-variety and let $T$ be a numerical semigroup. Then ${\mathcal P}_T = \{ S \cap T \mid S \in {\mathcal P}\}$ is an $R$-variety. Moreover, every $R$-variety is of this form.
\end{corollary}

\begin{proof}
	By Proposition~\ref{prop2}, we know that, if ${\mathcal P}$ is a pseudo-variety, then ${\mathcal P}$ is an $R$-variety. Thereby, by applying Corollary~\ref{cor9}, we conclude that ${\mathcal P}_T$ is an $R$-variety.
	
	Now, by Theorem~\ref{thm8}, we know that, if ${\mathcal R}$ is an $R$-variety, then there exist a variety ${\mathcal V}$ and a numerical semigroup $T$ such that ${\mathcal R} = {\mathcal V}_T$. To finish the proof, it is enough to observe that all varieties are pseudo-varieties.
\end{proof}    

Let us see an illustrative example of the above corollary.

\begin{example}\label{exmp11}
	From \cite[Example~7]{pseudovar}, we have the pseudo-variety
		$${\mathcal P} = \left\{ \langle 5,6,8,9 \rangle, \langle 5,6,9,13 \rangle, \langle 5,6,8 \rangle, \langle 5,6,13,14 \rangle, \right.$$ $$\left. \langle 5,6,9 \rangle, \langle 5,6,14 \rangle, \langle 5,6,13 \rangle, \langle 5,6,19 \rangle, \langle 5,6 \rangle \right\}.$$
	Thereby, we have that ${\mathcal P}_T$ is an $R$-variety for each numerical semigroup $T$.
\end{example}

\section{Monoids associated to an $R$-variety}\label{r-monoids}

In this section, ${\mathcal R}$ we will be an $R$-variety. Now, let $M$ be a submonoid of $({\mathbb N},+)$. We will say that $M$ is an ${\mathcal R}$-monoid if it is the intersection of elements of ${\mathcal R}$. The next result is easy to proof.

\begin{lemma}\label{lem12}
	The intersection of ${\mathcal R}$-monoids is an ${\mathcal R}$-monoid.
\end{lemma}

From the above lemma we have the following definition: let $A \subseteq \Delta({\mathcal R})$. We will say that ${\mathcal R}(A)$ is the \emph{${\mathcal R}$-monoid generated by $A$} if ${\mathcal R}(A)$ is equal to the intersection of all the ${\mathcal R}$-monoids which contain the set $A$. Observe that ${\mathcal R}(A)$ is the smallest ${\mathcal R}$-monoid which contains the set $A$ (with respect to the inclusion order). The next result has an easy proof too.

\begin{lemma}\label{lem13}
	If $A \subseteq \Delta({\mathcal R})$, then ${\mathcal R}(A)$ is equal to the intersection of all the elements of ${\mathcal R}$ which contain the set $A$.
\end{lemma}

Let us take $A \subseteq \Delta({\mathcal R})$. If $M={\mathcal R}(A)$, then we will say that $A$ is an \emph{${\mathcal R}$-system of generators} of $M$. Moreover,we will say that $A$ is a \emph{minimal ${\mathcal R}$-system of generators} of $M$ if $M\not={\mathcal R}(B)$ for all $B \varsubsetneq A$. The next purpose in this section will be to show that every ${\mathcal R}$-monoid has a unique minimal ${\mathcal R}$-system of generators. For that, we will give some previous lemmas. We can easily deduced the first one from Lemma~\ref{lem13}.

\begin{lemma}\label{lem14}
	Let $A,B$ be two subsets of $\Delta({\mathcal R})$ and let $M$ be an ${\mathcal R}$-monoid. We have that
	\begin{enumerate}
		\item if $A\subseteq B$, then ${\mathcal R}(A)\subseteq {\mathcal R}(B)$;
		\item ${\mathcal R}(A) = {\mathcal R}(\langle A\rangle)$;
		\item ${\mathcal R}(M)=M$.
	\end{enumerate}
\end{lemma}

If $M$ is an ${\mathcal R}$-monoid, then $M$ is a submonoid of $({\mathbb N},+)$. Moreover, as we commented in Section~\ref{allvarieties}, we know that there exists a finite subset $A$ of $M$ such that $M=\langle A \rangle$. Thereby, by applying Lemma~\ref{lem14}, we have that $M={\mathcal R}(M)={\mathcal R}(\langle A\rangle)={\mathcal R}(A)$. Consequently, $A$ is a finite ${\mathcal R}$-system of generators of $M$. Thus, we can establish the next result.

\begin{lemma}\label{lem15}
	Every ${\mathcal R}$-monoid has a finite ${\mathcal R}$-system of generators.
\end{lemma}

In the following result, we characterize the minimal ${\mathcal R}$-systems of generators.

\begin{lemma}\label{lem16}
	Let $A \subseteq \Delta({\mathcal R})$ and $M={\mathcal R}(A)$. Then $A$ is a minimal ${\mathcal R}$-system of generators of $M$ if and only if $a\notin {\mathcal R}(A\setminus\{a\})$ for all $a\in A$.
\end{lemma}

\begin{proof}
	\emph{(Necessity.)} If $a \in {\mathcal R}(A\setminus\{a\})$, then $A\subseteq {\mathcal R}(A\setminus\{a\})$. Thus, by Lemma~\ref{lem14}, we get that $M={\mathcal R}(A)\subseteq{\mathcal R}({\mathcal R}(A\setminus\{a\})) = {\mathcal R}(A\setminus\{a\}) \subseteq {\mathcal R}(A)=M$. Therefore, $M={\mathcal R}(A\setminus\{a\})$, in contradiction with the minimality of $A$.
	
	\emph{(Sufficiency.)} If $A$ is not a minimal ${\mathcal R}$-system of generators of $M$, then there exists $B \varsubsetneq A$ such that ${\mathcal R}(B)=M$. Then, by Lemma~\ref{lem14}, if $a\in A\setminus B$, then $a\in M={\mathcal R}(B)\subseteq {\mathcal R}(A\setminus\{a\})$, in contradiction with the hypothesis.
\end{proof}

The next result generalizes an evident property of submonoids of $({\mathbb N},+)$. More concretely, every element $x$ of a submonoid $M$ of $({\mathbb N},+)$ is expressible as a non-negative integer linear combination of the generators of $M$ that are smaller than or equal to $x$.

\begin{lemma}\label{lem17}
	Let $A \subseteq \Delta({\mathcal R})$ and $x\in {\mathcal R}(A)$. Then $x\in {\mathcal R}(\{a \in A \mid a\leq x\})$.
\end{lemma}

\begin{proof}
	Let us suppose that $x\not \in {\mathcal R}(\{a\in A\mid a\leq x\})$. Then, from Lemma~\ref{lem13}, we know that there exists $S\in {\mathcal R}$ such that $\{a\in A \mid a\leq x\}\subseteq S$ and $x\not \in S$. By applying now Lemma~\ref{lem7}, we have that $\bar{S}= S \cup \{m\in \Delta({\mathcal R}) \mid m \geq x+1 \} \in {\mathcal R}$. Observe that, obviously, $A\subseteq \bar{S}$ and $x\notin \bar{S}$. Therefore, by applying once again Lemma~\ref{lem13}, we get that $x\notin {\mathcal R}(A)$, in contradiction with the hypothesis.
\end{proof}

We are now ready to show the above announced result.

\begin{theorem}\label{thm18}
	Every ${\mathcal R}$-monoid admits a unique minimal ${\mathcal R}$-system of generators. In addition, such a ${\mathcal R}$-system is finite.
\end{theorem}

\begin{proof}
	Let $M$ be an ${\mathcal R}$-monoid and let $A,B$ be two minimal ${\mathcal R}$-systems of generators of $M$. We are going to see that $A=B$. For that, let us suppose that $A=\{a_1<a_2<\cdots\}$ and $B=\{b_1<b_2<\cdots\}$. If $A\neq B$, then there exists $i=\min\{k\mid a_k\neq b_k\}$. Let us assume, without loss of generality, that $a_i<b_i$. Since $a_i\in M={\mathcal R}(A)={\mathcal R}(B)$, by Lemma~\ref{lem17}, we have that $a_i\in {\mathcal R}(\{b_1,\ldots,b_{i-1}\})$. Because $\{b_1,\ldots,b_{i-1}\}=\{a_1,\ldots,a_{i-1}\}$, then $a_i\in {\mathcal R}(\{a_1,\ldots,a_{i-1}\})$, in contradiction with Lemma~\ref{lem16}. Finally, by Lemma~\ref{lem15}, we have that the minimal ${\mathcal R}$-system of generators is finite.
\end{proof}

If $M$ is a ${\mathcal R}$-monoid, then the cardinality of the minimal ${\mathcal R}$-system of generators of $M$ will be called the \emph{${\mathcal R}$-range} of $M$.

\begin{example}\label{exmp19}
	Let $S,T$ be two numerical semigroups such that $S \subseteq T$. We define recurrently the following sequence of numerical semigroups.
	\begin{itemize}
		\item $S_0=S$,
		\item if $S_i\neq T$, then $S_{i+1}=S_i\cup\{{\mathrm F}_T(S_i)\}$.
	\end{itemize}
	Since $T \setminus S$ is a finite set, then we get a finite chain of numerical semigroups $S=S_0 \varsubsetneq S_1\varsubsetneq \cdots \varsubsetneq S_n=T$. We will denote by ${\mathrm C}(S,T)$ the set $\{S_0,S_1,\ldots,S_n\}$ and will say that it is the \emph{chain of $S$ restricted to $T$}. It is clear that ${\mathrm C}(S,T)$ is an $R$-variety. Moreover, it is also clear that, for each $i\in\{1,\ldots,n\}$, $S_i$ is the smallest element of ${\mathrm C}(S,T)$ containing ${\mathrm F}_T(S_{i-1})$. Therefore, $\{{\mathrm F}_T(S_{i-1})\}$ is the minimal ${\mathrm C}(S,T)$-system of generators of $S_i$ for all $i\in\{1,\ldots,n\}$. Let us also observe that the empty set, $\emptyset$, is the minimal ${\mathrm C}(S,T)$-system of generators of $S_0$. Thereby, the ${\mathrm C}(S,T)$-range of $S_i$ is equal to $1$, if $i\in\{1,\ldots,n\}$, and $0$, if $i=0$.
\end{example}

It is well known that, if $M$ is a submonoid of $({\mathbb N},+)$ and $x\in M$, then $M\setminus \{x\}$ is another submonoid of $({\mathbb N},+)$ if and only if $x\in \mathrm{msg}(M)$. In the next result we generalize this property to ${\mathcal R}$-monoids.

\begin{proposition}\label{prop20}
	Let $M$ be an ${\mathcal R}$-monoid and let $x \in M$. Then $M \setminus \{x\}$ is an ${\mathcal R}$-monoid if and only if $x$ belongs to the minimal ${\mathcal R}$-system of generators of $M$.
\end{proposition}

\begin{proof}
	Let $A$ be the minimal ${\mathcal R}$-system of generators of $M$. If $x\not \in A$, then $A \subseteq M \setminus \{x\}$. Therefore, $M \setminus \{x\}$ is a ${\mathcal R}$-monoid containing $A$ and, consequently, $M={\mathcal R}(A) \subseteq M \setminus \{x\}$, which is a contradiction.
	
	Conversely, by Theorem~\ref{thm18}, we have that, if $x \in A$, then ${\mathcal R}(M \setminus \{x\}) \neq {\mathcal R}(A) = M$. Thereby, ${\mathcal R}(M \setminus \{x\})=M \setminus \{x\}$. Consequently, $M \setminus \{x\}$ is a ${\mathcal R}$-monoid.
\end{proof}

Let us illustrate the above proposition with an example.

\begin{example}\label{exmp21}
	Let $T$ be a numerical semigroup and let $A\subseteq T$. Then we know that ${\mathcal R}(A,T)=\{S \mid S \mbox{ is a numerical semigroup and } A\subseteq S\subseteq T \}$ is an $R$-variety. By applying Proposition~\ref{prop20}, we easily deduce that, if $S\in {\mathcal R}(A,T)$, then the minimal ${\mathcal R}(A,T)$-system of generators of $S$ is $\{x\in \mathrm{msg} \mid x\notin A \}$.
\end{example}

From Theorem~\ref{thm8} we know that every $R$-variety is of the form ${\mathcal V}_T =  \{S\cap T \mid S\in {\mathcal V} \}$, where ${\mathcal V}$ is a variety and $T$ is a numerical semigroup. Now, our purpose is to study the relation between ${\mathcal V}$-monoids and ${\mathcal V}_T$-monoids.

\begin{proposition}\label{prop22}
	Let $M$ be a submonoid of $({\mathbb N},+)$ and let $T$ be a numerical semigroup. Then $M$ is a ${\mathcal V}_T$-monoid if and only if there exists a ${\mathcal V}$-monoid $M'$ such that $M=M'\cap T$.
\end{proposition}

\begin{proof}
	\emph{(Necessity.)} If $M$ is a ${\mathcal V}_T$-monoid, then there exists ${\mathcal F} \subseteq {\mathcal V}_T$ such that $M=\bigcap_{S\in {\mathcal F}} S$. But, if $S\in {\mathcal F}$, then $S\in {\mathcal V}_T$ and, consequently, there exists $S' \in {\mathcal V}$ such that $S=S'\cap T$. Now, let ${\mathcal F}'=\{S'\in{\mathcal V} \mid S'\cap T \in {\mathcal F} \}$ and let $M'=\bigcap_{S'\in {\mathcal F}'} S'$. Then it is clear that $M'$ is a ${\mathcal V}$-monoid and that $M=M'\cap T$.
	
	\emph{(Sufficiency.)} If $M'$ is a ${\mathcal V}$-monoid, then there exists  ${\mathcal F}' \in {\mathcal V}$ such that $M'=\bigcap_{S'\in {\mathcal F}'} S'$. Let ${\mathcal F}=\{S'\cap T\mid S'\in {\mathcal F}' \}$. Then it is clear that ${\mathcal F} \subseteq {\mathcal V}_T$ and that $\bigcap_{S\in {\mathcal F}} S = M'\cap T$. Therefore, $M'\cap T$ is a ${\mathcal V}_T$-monoid.
\end{proof}

Observe that, as a consequence of the above proposition, we have that the set of ${\mathcal V}_T$-monoids is precisely given by $\left\{M\cap T \mid M \mbox{ is a } {\mathcal V}\mbox{-monoid} \right\}$.

\begin{corollary}\label{cor23}
	Let $T$ be a numerical semigroup. If $A\subseteq T$, then ${\mathcal V}_T(A)={\mathcal V}(A) \cap T$.
\end{corollary}

\begin{proof}
	By Proposition~\ref{prop22}, we know that ${\mathcal V}(A) \cap T$ is a${\mathcal V}_T$-monoid containing $A$. Therefore, ${\mathcal V}_T(A) \subseteq {\mathcal V}(A) \cap T$.
	
	Let us see now the opposite inclusion. By applying once more  Proposition~\ref{prop22}, we deduce that there exists a ${\mathcal V}$-monoid $M$ such that ${\mathcal V}_T(A)=M\cap T$. Thus, it is clear that $A\subseteq M$ and, thereby, ${\mathcal V}(A)\subseteq M$. Consequently, ${\mathcal V}(A) \cap T \subseteq M\cap T = {\mathcal V}_T(A)$.
\end{proof}

From Corollary~\ref{cor23}, we have that the set formed by the ${\mathcal V}_T$-monoids is $\left\{{\mathcal V}(A) \cap T \mid A \subseteq T \right\} = \left\{M\cap T \mid M\right.$ is a ${\mathcal V}$-monoid and its minimal ${\mathcal V}$-system of generators is including in $\left. T \right\}$. Moreover, observe that, if $T\in {\mathcal V}$, then ${\mathcal V}_T(A)={\mathcal V}(A)$ and, therefore, in such a case the set formed by all the ${\mathcal V}_T$-monoids coincides with the set formed by all the ${\mathcal V}$-monoids that are contained in $T$.

For some varieties there exist algorithms that allow us to compute ${\mathcal V}(A)$ by starting from $A$. Thereby, we can use such results in order to compute ${\mathcal V}_T (A)$. Let us see two examples of this fact.

\begin{example}\label{exmp24a}
	An $\mathrm{LD}$-semigroup (see \cite{digitales}) is a numerical semigroup $S$ fulfilling that $a+b-1\in S$ for all $a,b\in S\setminus \{0\}$. Let ${\mathcal V}$ the set formed by all $\mathrm{LD}$-semigroups. In \cite{digitales} it is shown that ${\mathcal V}$ is a variety. Let $T=\langle 5,7,9 \rangle$ (observe that $T\notin {\mathcal V}$). By Theorem~\ref{thm8}, we know that ${\mathcal V}_T=\{S\cap T \mid S\in{\mathcal V} \}$ is an $R$-variety. Let us suppose that we can compute ${\mathcal V}_T(\{5\})$.
	
	In \cite{digitales} we have an algorithm to compute ${\mathcal V}(A)$ by starting from $A$. By using such algorithm, in \cite[Example~33]{digitales} it is shown that ${\mathcal V}(\{5\})=\langle 5,9,13,17,21\rangle$. Therefore, by applying Corollary~\ref{cor23}, we have that ${\mathcal V}_T(\{5\})=\langle 5,9,13,17,21\rangle \cap \langle 5,7,9 \rangle = \langle 5,9,17,21 \rangle$.
\end{example}

\begin{example}\label{exmp24b}
	An $\mathrm{PL}$-semigroup (see \cite{frases}) is a numerical semigroup $S$ fulfilling that $a+b+1\in S$ for all $a,b\in S\setminus \{0\}$. Let ${\mathcal V}$ the set formed by all $\mathrm{PL}$-semigroups. In \cite{frases} it is shown that ${\mathcal V}$ is a variety and it is given an  algorithm to compute ${\mathcal V}(A)$ by starting from $A$. Let $T=\langle 4,7,13 \rangle$ (observe that $T\in {\mathcal V}$). By Theorem~\ref{thm8}, we know that ${\mathcal V}_T=\{S\cap T \mid S\in{\mathcal V} \}$ is an $R$-variety. Let us suppose that we can compute ${\mathcal V}_T(\{4,7\})$.
	
	From \cite[Example~48]{frases}, we know that ${\mathcal V}(\{4,7\})=\langle 4,7,9\rangle$. Thus, by applying Corollary~\ref{cor23}, we have that ${\mathcal V}_T(\{4,7\})=\langle 4,7,9\rangle \cap \langle 4,7,13 \rangle = \langle 4,7,13 \rangle$.
\end{example}

Let $T$ be a numerical semigroup. We know that, if $M$ is a ${\mathcal V}_T$-monoid, then there exists a ${\mathcal V}$-monoid, $M'$, with minimal ${\mathcal V}$-system of generators contained in $T$, such that $M=M'\cap T$. The next result says us that, in this situation, the minimal ${\mathcal V}$-system of generators of $M'$ is just the minimal ${\mathcal V}_T$-system of generators of $M$.

\begin{proposition}\label{prop25}
	Let $A\subseteq T$. Then $A$ is the minimal ${\mathcal V}_T$-system of generators of ${\mathcal V}_T(A)$ if and only if $A$ is the minimal ${\mathcal V}$-system of generators of ${\mathcal V}(A)$.
\end{proposition}

\begin{proof}
	\emph{(Necessity.)} Let us suppose that $A$ is not the minimal ${\mathcal V}$-system of generators of ${\mathcal V}(A)$. That is, there exists $B\varsubsetneq A$ such that ${\mathcal V}(B)={\mathcal V}(A)$. Then, from Corollary~\ref{cor23}, we have that ${\mathcal V}_T(A)={\mathcal V}(A) \cap T = {\mathcal V}(B)\cap T = {\mathcal V}_T(B)$. Therefore, $A$ is not the minimal ${\mathcal V}_T$-system of generators of ${\mathcal V}_T(A)$.
	
	\emph{(Sufficiency.)} Let us suppose that $A$ is not the minimal ${\mathcal V}_T$-system of generators of ${\mathcal V}_T(A)$. Then ${\mathcal V}_T(B)={\mathcal V}_T(A)$ for some subset $B\varsubsetneq A$. On the other hand, due to $A$ is the minimal ${\mathcal V}$-system of generators of ${\mathcal V}(A)$, from Lemma~\ref{lem16}, we have an element $a\in A$ such that $a\notin {\mathcal V}(B)$. Consequently, $a\in{\mathcal V}(A) \cap T$ and $a\notin {\mathcal V}(B) \cap T$. Finally, from Corollary~\ref{cor23}, ${\mathcal V}_T(A) = {\mathcal V}(A) \cap T \not= {\mathcal V}(B) \cap T ={\mathcal V}_T(B)$, which is a contradiction.
\end{proof}

We finish this section with two examples that illustrate the above proposition.

\begin{example}\label{exmp26a}
	 Let ${\mathcal V}$ be such as in Example~\ref{exmp24a} and let $T=\langle 4,6,7 \rangle$. From \cite[Example~26]{digitales}, we know that ${\mathcal V}(\{4,7,10\}) = \langle 4,7,10,13 \rangle$ and, moreover, that $\{4\}$ is its minimal ${\mathcal V}$-system of generators. Then, from Proposition~\ref{prop25}, $\{4\}$ is the minimal ${\mathcal V}$-system of generators of ${\mathcal V}_T(\{4,7,10\}) = \langle 4,7,10,13 \rangle \cap \langle 4,6,7 \rangle$.
\end{example}

\begin{example}\label{exmp26b}
	Let ${\mathcal V}$ be such as in Example~\ref{exmp24b} and let $T=\langle 3,4 \rangle$. From \cite[Example~44]{frases}, we know that $\{3\}$ is the minimal ${\mathcal V}$-system of generators os $S=\langle 3,7,11 \rangle$. Therefore, by Proposition~\ref{prop25}, $\{3\}$ is the minimal ${\mathcal V}_T$-system of generators of $S\cap T$.
\end{example}

\section{The tree associated to an $R$-variety}\label{tree}

Let $V$ be a non-empty set and let $E \subseteq \{(v,w) \in V \times V \mid v \neq w\}$. It is said that the pair $G=(V,E)$ is a \emph{graph}. In addition, the \emph{vertices} and \emph{edges} of $G$ are the elements of $V$ and $E$, respectively.

Let $x,y\in V$ and let us suppose that $(v_0,v_1),(v_1,v_2),\ldots,(v_{n-1},v_n)$ is a sequence of different edges such that $v_0=x$ and $v_n=y$. Then, it is said that such a sequence is a \emph{path (of length $n$)} connecting $x$ and $y$.

Let $G$ be a graph. Let us suppose that there exists $r$, vertex of $G$, such that it is connected with any other vertex $x$ by a unique path. Then it is said that $G$ is a \emph{tree} and that $r$ is its \emph{root}.

Let $x,y$ be vertices of a tree $G$ and let us suppose that there exists a path that connects $x$ and $y$. Then it is said that $x$ is a \emph{descendant} of $y$. Specifically, it is said that $x$ is a \emph{child} of $y$ when $(x,y)$ is an edge of $G$.

From now on in this section, let ${\mathcal R}$ denote an $R$-variety. We define the graph ${\mathrm G}({\mathcal R})$ in the following way,
\begin{itemize}
	\item ${\mathcal R}$ is the set of vertices of ${\mathrm G}({\mathcal R})$;
	\item $(S,S')\in {\mathcal R}\times {\mathcal R}$ is an edge of ${\mathrm G}({\mathcal R})$ if and only if $S'=S\cup\{{\mathrm F}_{\Delta({\mathcal R})}(S)\}$.
\end{itemize}

If $S \in {\mathcal R}$, then we can define recurrently (such as we did in Example~\ref{exmp19}) the sequence of elements in ${\mathcal R}$,
\begin{itemize}
	\item $S_0=S$,
	\item if $S_i\neq \Delta({\mathcal R})$, then $S_{i+1}=S_i\cup\{{\mathrm F}_{\Delta({\mathcal R})}(S_i)\}$.
\end{itemize}
Thus, we obtain a chain (of elements in ${\mathcal R}$) $S=S_0 \varsubsetneq S_1 \varsubsetneq \cdots \varsubsetneq S_n=\Delta({\mathcal R})$ such that $(S_i,S_{i+1})$ is an edge of ${\mathrm G}({\mathcal R})$ for all $i\in \{0,\ldots,n-1 \}$. We will denote by ${\mathrm C}_{\mathcal R}(S)$ the set $\{S_0,S_1,\ldots,S_n \}$ and will say that it is the \emph{chain of $S$ in ${\mathcal R}$}. The next result is easy to prove.

\begin{proposition}\label{prop27}
	${\mathrm G}({\mathcal R})$ is a tree with root $\Delta({\mathcal R})$.
\end{proposition}

Observe that, in order to recurrently construct ${\mathrm G}({\mathcal R})$ starting from $\Delta({\mathcal R})$, it is sufficient to compute the children of each vertex of ${\mathrm G}({\mathcal R})$. Let us also observe that, if $T$ is a child of $S$, then $S=T \cup \{{\mathrm F}_{\Delta({\mathcal R})}(T)\}$. Therefore, $T=S \setminus \{{\mathrm F}_{\Delta({\mathcal R})}(T)\}$. Thus, if $T$ is a child of $S$, then there exists an integer $x > {\mathrm F}_{\Delta({\mathcal R})}(S)$ such that $T=S \setminus \{x\}$. As a consequence of Propositions~\ref{prop20} and \ref{prop27}, and defining ${\mathrm F}_{\Delta({\mathcal R})}(\Delta({\mathcal R}))=-1$, we have the following result.

\begin{theorem}\label{thm28}
	The graph ${\mathrm G}({\mathcal R})$ is a tree with root equal to $\Delta({\mathcal R})$. Moreover, the set formed by the children of a vertex $S \in {\mathcal R}$ is $\left\{S \setminus \{x\} \mid x \right.$ is an element of the minimal ${\mathcal R}$-system of generators of $S$ and $\left. x>{\mathrm F}_{\Delta({\mathcal R})}(S) \right\}$.
\end{theorem}

We can reformulate the above theorem in the following way.

\begin{corollary}\label{cor29}
	The graph ${\mathrm G}({\mathcal R})$ is a tree with root equal to $\Delta({\mathcal R})$. Moreover, the set formed by the children of a vertex $S \in {\mathcal R}$ is $\left\{S \setminus \{x\} \mid x \in \mathrm{msg}(S), \right.$ $\left. x>{\mathrm F}_{\Delta({\mathcal R})}(S) \mbox{ and } S\setminus \{x\} \in {\mathcal R} \right\}$.
\end{corollary}

We illustrate the previous results with an example.

\begin{example}\label{exmp30}
	We are going to build the $R$-variety ${\mathcal R} = \left[ \langle 5,6 \rangle, \langle 5,6,7 \rangle \right] = \left\{ S \mid S \right.$ is a numerical semigroup and $\left. \langle 5,6 \rangle \subseteq S \subseteq \langle 5,6,7 \rangle \right\}$. Observe that, if $S\in {\mathcal R}$ and $x \in \mathrm{msg}(S)$, then $S\setminus \{x\} \in {\mathcal R}$ if and only if $x\notin \{5,6\}$. Moreover, the maximum of ${\mathcal R}$ is $\Delta = \langle 5,6,7 \rangle$. By applying Corollary~\ref{cor29}, we can recurrently build ${\mathrm G}({\mathcal R})$ in the following way.
	\begin{itemize}
		\item $\langle 5,6,7 \rangle$ has got a unique child, which is $\langle 5,6,7 \rangle \setminus \{7\} = \langle 5,6,13,14 \rangle$. Moreover, ${\mathrm F}_{\Delta}(\langle 5,6,13,14 \rangle)=7$.
		\item $\langle 5,6,13,14 \rangle$ has got two children, which are $\langle 5,6,13,14 \rangle \setminus \{13\} = \langle 5,6,14 \rangle$ and $\langle 5,6,13,14 \rangle \setminus \{14\} = \langle 5,6,13 \rangle$. Moreover, ${\mathrm F}_{\Delta}(\langle 5,6,14 \rangle)=13$ and ${\mathrm F}_{\Delta}(\langle 5,6,13 \rangle)=14$.
		\item $\langle 5,6,13 \rangle$ has not got children.
		\item $\langle 5,6,14 \rangle$ has got a unique child, which is $\langle 5,6,14 \rangle \setminus \{14\} = \langle 5,6,19 \rangle$. Moreover, ${\mathrm F}_{\Delta}(\langle 5,6,19 \rangle)=14$.
		\item $\langle 5,6,19 \rangle$ has got a unique child, which is $\langle 5,6,19 \rangle \setminus \{19\} = \langle 5,6 \rangle$. Moreover, ${\mathrm F}_{\Delta}(\langle 5,6 \rangle)=19$.
		\item $\langle 5,6 \rangle$ has not got children.
	\end{itemize}
	Therefore, in this situation, ${\mathrm G}({\mathcal R})$ is given by the next diagram.
	
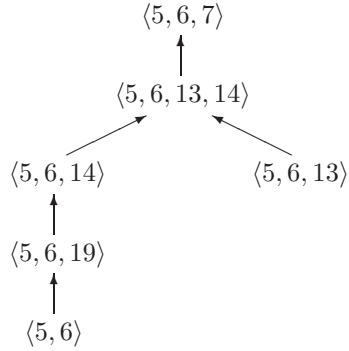
\begin{center}
		\begin{picture}(144,125)
		\put(58,120){$\langle 5,6,7 \rangle$}
		\put(73,100){\vector(0,1){15}}
		\put(48,90){$\langle 5,6,13,14 \rangle$}
		\put(30,70){\vector(2,1){30}} \put(115,70){\vector(-2,1){30}}
		\put(8,60){$\langle 5,6,14 \rangle$} \put(100,60){$\langle 5,6,13 \rangle$}
		\put(25,40){\vector(0,1){15}}
		\put(8,30){$\langle 5,6,19 \rangle$}
		\put(25,10){\vector(0,1){15}}
		\put(14,0){$\langle 5,6 \rangle$}
		\end{picture}
	\end{center}
	Observe that, if we represent the vertices of ${\mathrm G}({\mathcal R})$ using their minimal ${\mathcal R}$-systems of generators, then we have that ${\mathrm G}({\mathcal R})$ is given by the following diagram.
	
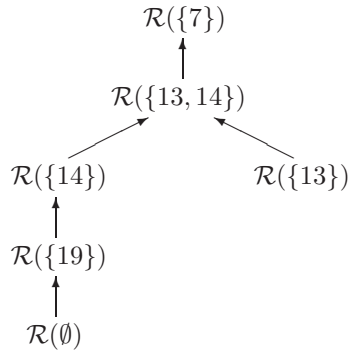
\begin{center}
		\begin{picture}(144,125)
		\put(58,120){${\mathcal R}(\{7\})$}
		\put(73,100){\vector(0,1){15}}
		\put(46,90){${\mathcal R}(\{13,14\})$}
		\put(30,70){\vector(2,1){30}} \put(115,70){\vector(-2,1){30}}
		\put(8,60){${\mathcal R}(\{14\})$} \put(100,60){${\mathcal R}(\{13\})$}
		\put(25,40){\vector(0,1){15}}
		\put(8,30){${\mathcal R}(\{19\})$}
		\put(25,10){\vector(0,1){15}}
		\put(14,0){${\mathcal R}(\emptyset)$}
		\end{picture}
	\end{center}
\end{example}

Let us observe that the $R$-variety ${\mathcal R} = \left[ \langle 5,6 \rangle, \langle 5,6,7 \rangle \right]$ depict in the above example is finite and, therefore, we have been able to build all its elements in a finite number of steps. If the $R$-variety is infinite, then it is not possible such situation. However, as a consequence of Theorem~\ref{thm28}, we can show an algorithm in order to compute all the elements of the $R$-variety when the genus is fixed.

\begin{algorithm}\label{alg31}
	{\textrm
		INPUT: A positive integer $g$. \par
		OUTPUT: $\{S\in{\mathcal R} \mid {\mathrm g}(S)=g \}$. \par
		(1) If $g < {\mathrm g}(\Delta({\mathcal R}))$, then return $\emptyset$. \par
		(2) Set $A=\{\Delta({\mathcal R})\}$ and $i={\mathrm g}(\Delta({\mathcal R}))$. \par
		(3) If $i=g$, then return $A$. \par
		(4) For each $S \in A$, compute the set $B_S$ formed by all elements of the minimal ${\mathcal R}$-system of generators of $S$ that are greater than ${\mathrm F}_{\Delta({\mathcal R})}(S)$. \par
		(5) If ${\displaystyle \bigcup_{S\in A} B_S} = \emptyset$, then return $\emptyset$. \par
		(6) Set ${\displaystyle A=\bigcup_{S\in A} \big\{S \setminus \{x\} \mid x \in B_S \big\} }$, $i=i+1$, and go to (3).
	}
\end{algorithm}

We illustrate the operation of this algorithm with an example.

\begin{example}\label{exmp32}
	Let $\Delta=\langle 4,6,7 \rangle = \{0,4,6,7,8,10,\to\}$. It is clear that ${\mathrm g}(\Delta)=5$. Let ${\mathcal R} = \left\{ S \mid S \right.$ is a numerical semigroup and $\left. \{4,6\} \subseteq S \subseteq \Delta \right\}$. We have that ${\mathcal R}$ is an infinite $R$-variety because $\langle 4,6,2k+1 \rangle \in {\mathcal R}$ for all $k\in\{5,\to\}$. By using Algorithm~\ref{alg31}, we are going to compute the set $\left\{S\in{\mathcal R} \mid {\mathrm g}(S)=8 \right\}$.
	\begin{itemize}
		\item $A=\{\Delta\}$, $i=5$.
		\item $B_{\Delta}=\{7\}$.
		\item $A=\{\langle 4,6,11,13 \rangle\}$, $i=6$.
		\item $B_{\langle 4,6,11,13 \rangle}=\{11,13\}$.
		\item $A=\{\langle 4,6,13,15 \rangle, \langle 4,6,11 \rangle \}$, $i=7$.
		\item $B_{\langle 4,6,13,15 \rangle}=\langle 13,15 \rangle$ and  $B_{\langle 4,6,11 \rangle}=\emptyset$.
		\item $A=\{\langle 4,6,15,17 \rangle, \langle 4,6,13 \rangle \}$, $i=8$.
		\item The algorithm returns $\{\langle 4,6,15,17 \rangle, \langle 4,6,13 \rangle \}$.
	\end{itemize}
\end{example}

Our next purpose in this section will be to show that, if ${\mathcal R}$ is an $R$-variety and $T\in {\mathcal R}$, then the set formed by all the descendants of $T$ in the tree ${\mathrm G}({\mathcal R})$ is also an $R$-variety. It is clear that, if $S,T\in {\mathcal R}$, then $S$ is a descendant of $T$ if and only if $T\in {\mathrm C}_{\mathcal R}(S)$. Therefore, we can establish the following result.

\begin{lemma}\label{lem33}
	Let ${\mathcal R}$ be an $R$-variety and $S,T\in {\mathcal R}$. Then $S$ is a descendant of $T$ if and only if there exists $n\in {\mathbb N}$ such that $T=S\cup \left\{x\in \Delta({\mathcal R}) \mid x\geq n \right\}$.
\end{lemma}

As an immediate consequence of the above lemma, we have the next one.

\begin{lemma}\label{lem34}
	Let ${\mathcal R}$ be an $R$-variety and $S,T\in {\mathcal R}$ such that $S\not=T$. If $S$ is a descendant of $T$, then ${\mathrm F}_{\Delta({\mathcal R})}(S)={\mathrm F}_T(S)$.
\end{lemma}

Now we are ready to show the announced result.

\begin{theorem}\label{thm35}
	Let ${\mathcal R}$ be an $R$-variety and $T\in {\mathcal R}$. Then 
	$${\mathcal D}(T) = \left\{S\in {\mathcal R} \mid S \mbox{ is a descendant of $T$ in the tree } {\mathrm G}({\mathcal R}) \right\}$$
	is an $R$-variety.
\end{theorem}

\begin{proof}
	Clearly, $T$ is the maximum of ${\mathcal D}(T)$. Let us see that, if $S_1,S_2 \in {\mathcal D}(T)$, then $S_1 \cap S_2 \in {\mathcal D}(T)$. Since, from Lemma~\ref{lem33}, we know that there exist $n_1,n_2 \in {\mathbb N}$ such that $T=S_i\cup \left\{x\in \Delta({\mathcal R}) \mid x\geq n_i \right\}$, $i=1,2$, it is sufficient to show that $T=(S_1\cap S_2) \cup \left\{x\in \Delta({\mathcal R}) \mid x\geq \min\{n_1,n_2\} \right\}$. It is obvious that $(S_1\cap S_2) \cup \left\{x\in \Delta({\mathcal R}) \mid x\geq \min\{n_1,n_2\} \right\} \subseteq T$. Let us see now the opposite inclusion. For that, let $t\in T$ such that $t\notin S_1\cap S_2$. Then $t\notin S_1$ or $t\notin S_2$ and, therefore, $t\in \left\{x\in \Delta({\mathcal R}) \mid x\geq n_1 \right\}$ or $t\in \left\{x\in \Delta({\mathcal R}) \mid x\geq n_2 \right\}$. Thereby, $t\in \left\{x\in \Delta({\mathcal R}) \mid x\geq \min\{n_1,n_2\} \right\}$. Consequently, $T\subseteq (S_1\cap S_2) \cup \left\{x\in \Delta({\mathcal R}) \mid x\geq \min\{n_1,n_2\} \right\}$. By applying again Lemma~\ref{lem33}, we can assert that $S_1\cap S_2\in {\mathcal D}(T)$.
	
	Finally, let $S\in {\mathcal D}(T)$ such that $S\not= T$. From Lemma~\ref{lem34}, $S\cup \{{\mathrm F}_T(S)\} = S \cup \{{\mathrm F}_{\Delta({\mathcal R})}(S)\} \in {\mathcal R}$ and, in consequence, $S\cup \{{\mathrm F}_T(S)\} \in {\mathcal D}(T)$.
\end{proof}

From the previous comment to \cite[Example 7]{pseudovar}, we know that, if ${\mathcal V}$ is a variety and $T \in {\mathcal V}$, then ${\mathcal D}(T)$ is a pseudo-variety and, moreover, every pseudo-variety can be obtained in this way. Therefore, there exist $R$-varieties which are not the set formed by all the descendants of an element belonging to a variety.

The following result shows that an $R$-variety can be obtained as the set formed by intersecting all the descendants, of an element belonging to a variety, with a numerical semigroup. 

\begin{corollary}\label{cor36}
	Let ${\mathcal V}$ be a variety, let $\Delta \in {\mathcal V}$, and let $T$ be a numerical semigroup. Let ${\mathcal D}(\Delta) = \left\{S \mid S \mbox{ is a descendant of } \Delta \mbox{ in } {\mathrm G}({\mathcal V}) \right\}$ and let ${\mathcal D}(\Delta,T) = \left\{S\cap T \mid S \in {\mathcal D}(\Delta) \right\}$. Then ${\mathcal D}(\Delta,T)$ is an $R$-variety. Moreover, every $R$-variety can be obtained in this way.
\end{corollary}

\begin{proof}
	If ${\mathcal V}$ is a variety, then ${\mathcal V}$ is an $R$-variety and, by applying Theoem~\ref{thm35}, we have that ${\mathcal D}(\Delta)$ is an $R$-variety. From Corollary~\ref{cor9}, we conclude that ${\mathcal D}(\Delta,T)$ is an $R$-variety.
	
	If ${\mathcal R}$ is an $R$-variety, by Theorem~\ref{thm8}, we know that there exist a variety ${\mathcal V}$ and a numerical semigroup $T$ such that ${\mathcal R}=\{S\cap T \mid S\in {\mathcal V} \}$. Now, it is clear that ${\mathcal V}={\mathcal D}({\mathbb N}) = \left\{S \mid S \mbox{ is a descendant of ${\mathbb N}$ in } {\mathrm G}({\mathcal V}) \right\}$. Therefore, we have that ${\mathcal R} = \left\{S\cap T \mid S\in {\mathcal D}({\mathbb N}) \right\} = {\mathcal D}({\mathbb N},T)$.
\end{proof}

In the next result we see that the above corollary is also true when we write pseudo-variety instead of variety.

\begin{corollary}\label{cor37}
	Let ${\mathcal P}$ be a pseudo-variety, let $\Delta \in {\mathcal P}$, and let $T$ be a numerical semigroup. Let ${\mathcal D}(\Delta) = \left\{S \mid S \mbox{ is a descendant of } \Delta \mbox{ in } {\mathrm G}({\mathcal P}) \right\}$ and let ${\mathcal D}(\Delta,T) = \left\{S\cap T \mid S \in {\mathcal D}(\Delta) \right\}$. Then ${\mathcal D}(\Delta,T)$ is an $R$-variety. Moreover, every $R$-variety can be obtained in this way.
\end{corollary}

\begin{proof}
	If ${\mathcal P}$ is a pseudo-variety, then ${\mathcal P}$ is an $R$-variety and, by applying Theorem~\ref{thm35}, we have that ${\mathcal D}(\Delta)$ is an $R$-variety as well. Now, from Corollary~\ref{cor9}, we have that ${\mathcal D}(\Delta,T)$ is an $R$-variety.
	
	That every $R$-variety can be obtained in this way is an immediate consequence of Corollary~\ref{cor36} and having in mind that each variety is a pseudo-variety.
\end{proof}

We conclude this section by illustrating the above corollary with an example.

\begin{example}\label{exmp38}
	Let ${\mathcal P}$ the pseudo-variety which appear in Example~\ref{exmp11}. In \cite[Example~7]{pseudovar} it is shown that ${\mathrm G}({\mathcal P})$ is given by the next subtree.
	
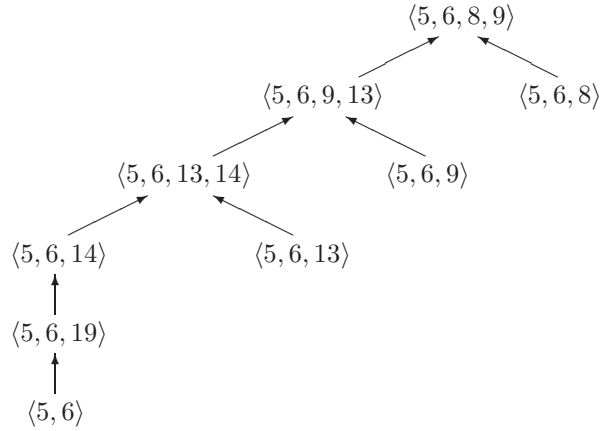
\begin{center}
		\begin{picture}(239,155)
		\put(158,150){$\langle 5,6,8,9 \rangle$}
		\put(140,130){\vector(2,1){30}} \put(215,130){\vector(-2,1){30}}
		\put(103,120){$\langle 5,6,9,13 \rangle$} \put(200,120){$\langle 5,6,8 \rangle$}
		\put(85,100){\vector(2,1){30}} \put(165,100){\vector(-2,1){30}}
		\put(48,90){$\langle 5,6,13,14 \rangle$} \put(150,90){$\langle 5,6,9 \rangle$}
		\put(30,70){\vector(2,1){30}} \put(115,70){\vector(-2,1){30}}
		\put(8,60){$\langle 5,6,14 \rangle$} \put(100,60){$\langle 5,6,13 \rangle$}
		\put(25,40){\vector(0,1){15}}
		\put(8,30){$\langle 5,6,19 \rangle$}
		\put(25,10){\vector(0,1){15}}
		\put(14,0){$\langle 5,6 \rangle$}
		\end{picture}
	\end{center}
	
	By applying Corollary~\ref{cor37}, we have that, if $T$ is a numerical semigroup, then ${\mathcal R} = \left\{S\cap T \mid S\in{\mathcal D}(\langle 5,6,13,14 \rangle) \right\}$ is an $R$-variety. Let us observe that ${\mathcal D}(\langle 5,6,13,14 \rangle) =\left\{ \langle 5,6,13,14 \rangle \right.$, $\langle 5,6,14 \rangle$, $\langle 5,6,13 \rangle$, $\langle 5,6,19 \rangle$, $\left. \langle 5,6\rangle \right\}$.
\end{example}

\section{The smallest $R$-variety containing a family of numerical semigroups}\label{r-var-fns}

In \cite[Proposition~2]{variedades} it is proved that the intersection of varieties is a variety. As a consequence of this, we have that there exists the smallest variety which contains a given family of numerical semigroups.

On the other hand, in \cite{pseudovar} was shown that, in general, the intersection of pseudo-varieties is not a pseudo-variety. Nevertheless, in \cite[Theorem~4]{pseudovar} it is proved that there exists the smallest pseudo-variety which contains a given family of numerical semigroups.

Our first objective in this section will be to show that, in general, we cannot talk about the smallest $R$-variety which contains a given family of numerical semigroups.

\begin{lemma}\label{lem39}
	Let ${\mathcal F}$ be a family of numerical semigroups and let $\Delta$ be a numerical semigroup such that $S\subseteq \Delta$ for all $S \in {\mathcal F}$. Then there exists an $R$-variety ${\mathcal R}$ such that ${\mathcal F} \subseteq {\mathcal R}$ and $\max({\mathcal R}) = \Delta$.
\end{lemma}

\begin{proof}
	Let ${\mathcal R}=\left\{S \mid S \right.$ is a numerical semigroup and $\left. S \subseteq \Delta \right\}$. From Item~1 in Example~\ref{exmp3}, we have that ${\mathcal R}$ is an $R$-variety. Now, it is trivial that ${\mathcal F} \subseteq {\mathcal R}$ and $\max({\mathcal R}) = \Delta$.
\end{proof}

The proof of the next lemma is straightforward and we can omit it.

\begin{lemma}\label{lem40}
	Let $\{{\mathcal R}_i\}_{i\in I}$ be a family of $R$-varieties such that $\max({\mathcal R}_i)=\Delta$ for all $i \in I$. Then $\bigcap_{i \in I} {\mathcal R}_i$ is an $R$-variety and $\max\left(\bigcap_{i \in I} {\mathcal R}_i\right)=\Delta$.
\end{lemma}

The following result says us that there exists the smallest $R$-variety which contains a given family of numerical semigroups and has a certain maximum.

\begin{proposition}\label{prop41}
	Let ${\mathcal F}$ be a family of numerical semigroups and let $\Delta$ be a numerical semigroup such that $S\subseteq \Delta$ for all $S \in {\mathcal F}$. Then there exists the smallest $R$-variety which contains ${\mathcal F}$ and with maximum equal to $\Delta$.
\end{proposition}

\begin{proof}
	Let ${\mathcal R}$ be the intersection of all the $R$-varieties containing ${\mathcal F}$ and with maximum equal to $\Delta$. From Lemmas~\ref{lem39} and \ref{lem40} we have the conclusion.
\end{proof}

We will denote by ${\mathcal R}({\mathcal F},\Delta)$ the $R$-variety given by Proposition~\ref{prop41}. Now we are interested in describe the elements of such an $R$-variety.

\begin{lemma}\label{lem42}
	Let $S_1,S_2,\ldots,S_n,\Delta$ be numerical semigroups such that $S_i \subseteq \Delta$ for all $i\in \{1,\ldots,n \}$. Then ${\mathrm F}_\Delta(S_1\cap\cdots\cap S_n) = \max\left\{ {\mathrm F}_\Delta(S_1),\ldots,{\mathrm F}_\Delta(S_n) \right\}$.
\end{lemma}

\begin{proof}
	We have that ${\mathrm F}_\Delta(S_1\cap\cdots\cap S_n) = \max\left(\Delta\setminus (S_1\cap\cdots\cap S_n) \right) =$
	
	\noindent $\max\left( (\Delta\setminus S_1) \cup\cdots\cup (\Delta\setminus S_n) \right) = \max\left\{ \max(\Delta\setminus S_1), \ldots, \max(\Delta\setminus S_n) \right\} =$
	
	\noindent $\max\left\{ {\mathrm F}_\Delta(S_1),\ldots,{\mathrm F}_\Delta(S_n) \right\}.$
\end{proof}

Let us recall that, if $S$ and $\Delta$ are numerical semigroups such that $S\subseteq \Delta$, then we defined ${\mathrm C}(S,\Delta)$ in Example~\ref{exmp19} (that is, the chain of $S$ restricted to $\Delta$). If ${\mathcal F}$ is a family of numerical semigroups such that $S\subseteq \Delta$ for all $S \in {\mathcal F}$, then we will denote by ${\mathrm C}({\mathcal F},\Delta)$ the set $\bigcup_{S \in {\mathcal F}}{\mathrm C}(S,\Delta)$.

\begin{theorem}\label{thm43}
	Let ${\mathcal F}$ be a family of numerical semigroups and let $\Delta$ be a numerical semigroup such that $S\subseteq \Delta$ for all $S\in {\mathcal F}$. Then ${\mathcal R}({\mathcal F},\Delta)$ is the set formed by all the finite intersections of elements in ${\mathrm C}({\mathcal F},\Delta)$.
\end{theorem}

\begin{proof}
	Let ${\mathcal R} = \left\{S_1\cap \cdots \cap S_n \mid n\in {\mathbb N} \setminus \{0\} \mbox{ and } S_1,\ldots,S_n \in {\mathrm C}({\mathcal F},\Delta) \right\}$. Having in mind that ${\mathcal R}({\mathcal F},\Delta)$ is an $R$-variety which contains ${\mathcal F}$ and with maximum equal to $\Delta$, we easily deduce that ${\mathcal R} \subseteq {\mathcal R}({\mathcal F},\Delta)$.
	
	Let us see now that ${\mathcal R}$ is an $R$-variety. On the one hand, it is clear that $\Delta=\max({\mathcal R})$ and that, if $S,T\in {\mathcal R}$, then $S\cap T \in {\mathcal R}$. On the other hand, let $S\in {\mathcal R}$ such that $S\not= \Delta$. Then $S=S_1\cap \cdots\cap S_n$ for some $S_1,\ldots,S_n \in {\mathrm C}({\mathcal F},\Delta)$. Now, from Lemma~\ref{lem42}, we have that ${\mathrm F}_\Delta(S) = \max\left\{ {\mathrm F}_\Delta(S_1),\ldots,{\mathrm F}_\Delta(S_n) \right\}$ and, thus, ${\mathrm F}_\Delta(S_i) \leq {\mathrm F}_\Delta(S)$ for all $i\in \{1,\ldots,n\}$. Let us observe that, if ${\mathrm F}_\Delta(S) > {\mathrm F}_\Delta(S_i)$, then $S_i \cup \{{\mathrm F}_\Delta(S)\} = S_i$. Moreover, if ${\mathrm F}_\Delta(S) = {\mathrm F}_\Delta(S_i)$, then we get $S_i \cup \{{\mathrm F}_\Delta(S)\} = S_i \cup \{{\mathrm F}_\Delta(S_i)\} \in {\mathrm C}({\mathcal F},\Delta)$. Therefore, $S_i \cup \{{\mathrm F}_\Delta(S)\} \in {\mathrm C}({\mathcal F},\Delta)$ for all $i\in \{1,\ldots,n\}$. Since $S \cup \{{\mathrm F}_\Delta(S)\} =(S_1 \cup \{{\mathrm F}_\Delta(S)\}) \cap \cdots \cap (S_n \cup \{{\mathrm F}_\Delta(S)\})$, then $S \cup \{{\mathrm F}_\Delta(S)\} \in {\mathcal R}$. Consequently, ${\mathcal R}$ is an $R$-variety.
	
	Finally, since ${\mathcal R}$ is an $R$-variety which contains ${\mathcal F}$ and with maximum equal to $\Delta$, then ${\mathcal R}({\mathcal F},\Delta) \subseteq {\mathcal R}$ and, thereby, we conclude that ${\mathcal R}={\mathcal R}({\mathcal F},\Delta)$.
\end{proof}

Let us observe that, if ${\mathcal F}$ is a finite family, then ${\mathrm C}({\mathcal F},\Delta)$ is a finite set and, therefore, ${\mathcal R}({\mathcal F},\Delta)$ is a finite $R$-variety.

\begin{lemma}\label{lem44}
	Let ${\mathcal R}$ and ${\mathcal R}'$ be two $R$-varieties. If ${\mathcal R} \subseteq {\mathcal R}'$, then $\Delta({\mathcal R}) \subseteq \Delta({\mathcal R}')$.
\end{lemma}

\begin{proof}
	If ${\mathcal R} \subseteq {\mathcal R}'$, then $\Delta({\mathcal R}) \in {\mathcal R}'$. Therefore, $\Delta({\mathcal R}) \subseteq \Delta({\mathcal R}')$.
\end{proof}

The next example shows us that, in general, we cannot talk about the smallest $R$-variety which contains a given family of numerical semigroups.

\begin{example}\label{exmp45}
	Let ${\mathcal F}=\left\{\langle 5,6 \rangle, \langle 5,7 \rangle \right\}$. As a consequence of Lemma~\ref{lem44}, the candidate to be the smallest $R$-variety which contains ${\mathcal F}$ must have as maximum the numerical semigroup $\langle 5,6,7 \rangle$ (that is, the smallest numerical semigroup containing $\langle 5,6 \rangle$ and $\langle 5,7 \rangle$). Thus, the candidate to be the smallest $R$-variety which contains ${\mathcal F}$ is ${\mathcal R}({\mathcal F},\langle 5,6,7 \rangle)$.
	
	Let us see now that ${\mathcal R}({\mathcal F},\langle 5,6,7 \rangle) \not\subseteq {\mathcal R}({\mathcal F},\langle 5,6,7,8 \rangle)$ and, in this way, that there does not exist the smallest $R$-variety which contains ${\mathcal F}$. In order to do it, we will show that $\langle 5,6,7 \rangle \notin {\mathcal R}({\mathcal F},\langle 5,6,7,8 \rangle)$. In fact, by applying Theorem~\ref{thm43}, if $\langle 5,6,7 \rangle \in {\mathcal R}({\mathcal F},\langle 5,6,7,8 \rangle)$, then we deduce that there exist $S_1 \in {\mathrm C}(\langle 5,6 \rangle, \langle 5,6,7,8 \rangle)$ and $S_2 \in {\mathrm C}(\langle 5,7 \rangle, \langle 5,6,7,8 \rangle)$ such that $S_1 \cap S_2 = \langle 5,6,7 \rangle$. Since $S_1 \in {\mathrm C}(\langle 5,6 \rangle, \langle 5,6,7,8 \rangle)$, then there exists $n_1\in {\mathbb N}$ such that $S_1=\langle 5,6 \rangle \cup \{x \in \langle 5,6,7,8 \rangle \mid x \geq n_1\}$. Moreover, $\langle 5,6,7 \rangle \subseteq S_1$ and, thereby, $n_1\leq 7$. Consequently, $8\in S_1$. By an analogous reasoning, we have that $8\in S_2$ too. Consequently, $8\in S_1 \cap S_2 = \langle 5,6,7 \rangle$, which is false.
\end{example}

Let ${\mathcal R}$ be an $R$-variety. We will say that ${\mathcal F}$ (subset of ${\mathcal R}$) is a \emph{system of generators} of ${\mathcal R}$ if ${\mathcal R} = {\mathcal R}({\mathcal F},\Delta({\mathcal R}))$. Let us observe that, in such a case, ${\mathcal R}$ is the smallest $R$-variety which contains ${\mathcal F}$ and with maximum equal to $\Delta({\mathcal R})$.

We will say that an $R$-variety, ${\mathcal R}$, is finitely generated if there exists a finite set ${\mathcal F} \subseteq {\mathcal R}$ such that ${\mathcal R} = {\mathcal R}({\mathcal F},\Delta({\mathcal R}))$ (that is, if ${\mathcal R}$ has a finite system of generators). As a consequence of Theorem~\ref{thm43}, we have the following result.

\begin{corollary}\label{cor46}
	An $R$-variety is finitely generated if and only if it is finite.
\end{corollary}

From now on, ${\mathcal F}$ will denote a family of numerical semigroups and $\Delta$ will denote a numerical semigroup such that $S\subseteq \Delta$ for all $S\in {\mathcal F}$. Our purpose is to give a method in order to compute the minimal ${\mathcal R}({\mathcal F},\Delta)$-system of generators of a ${\mathcal R}({\mathcal F},\Delta)$-monoid by starting from ${\mathcal F}$ and $\Delta$.

If $A\subseteq \Delta$, then for each $S \in {\mathcal F}$ we define
$$\alpha(S) = \left\{ \begin{array}{l} S, \, \mbox{ if } A \subseteq S, \\[3pt] S \cup \{x\in \Delta \mid x\geq x_S \}, \, \mbox{ if } A \nsubseteq S, \end{array} \right.$$
where $x_S = \min\{ a\in A \mid a \notin S\}$.

As a consequence of Lemma~\ref{lem13} and Theorem~\ref{thm43}, we have the next result.

\begin{lemma}\label{lem47}
	The ${\mathcal R}({\mathcal F},\Delta)$-monoid generated by $A$ is $\bigcap_{S \in {\mathcal F}} \alpha(S)$.
\end{lemma}

Recalling that ${\mathcal R}({\mathcal F},\Delta)(A)$ denotes the ${\mathcal R}({\mathcal F},\Delta)$-monoid generated by $A$, we have the following result.

\begin{proposition}\label{prop48}
	If $A\subseteq \Delta$, then $B=\left\{x_S \mid S \in{\mathcal F} \mbox{ and } A\not\subseteq S \right\}$ is the minimal ${\mathcal R}({\mathcal F},\Delta)$-system of generators of ${\mathcal R}({\mathcal F},\Delta)(A)$.
\end{proposition}

\begin{proof}
	Let us observe that, if $S\in{\mathcal F}$, then $A\subseteq S$ if and only if $B \subseteq S$. Moreover, if $A\not\subseteq S$, then $\min\{a\in A \mid a \notin S \} = \min\{b\in B \mid b \notin S \}$. Therefore, by applying Lemma~\ref{lem47}, we have that ${\mathcal R}({\mathcal F},\Delta)(A)={\mathcal R}({\mathcal F},\Delta)(B)$. Consequently, in order to prove that $B$ is the minimal ${\mathcal R}({\mathcal F},\Delta)$-system of generators of ${\mathcal R}({\mathcal F},\Delta)(A)$, will be enough to see that, if $C\varsubsetneq B$, then
	${\mathcal R}({\mathcal F},\Delta)(C) \not= {\mathcal R}({\mathcal F},\Delta)(A)$.
	
	In effect, if $C\varsubsetneq B$, then there exists $S\in {\mathcal F}$ such that $x_S \notin C$ and, thereby, we have that $C\subseteq S$ or that $\min\{c\in C \mid c\notin S \} > x_S$. Now, by applying once more time Lemma~\ref{lem47}, we easily deduce that $x_S \notin  {\mathcal R}({\mathcal F},\Delta)(C)$. Since $x_S \in B \subseteq A$, then we get that $A \not\subseteq {\mathcal R}({\mathcal F},\Delta)(C)$ and, therefore, ${\mathcal R}({\mathcal F},\Delta)(C) \not= {\mathcal R}({\mathcal F},\Delta)(A)$.
\end{proof}

As an immediate consequence of the above proposition we have the next result.

\begin{corollary}\label{cor49}
	Every ${\mathcal R}({\mathcal F},\Delta)$-monoid has ${\mathcal R}({\mathcal F},\Delta)$-range less than or equal to the cardinality of ${\mathcal F}$.
\end{corollary}

We will finish this section by illustrating its content with an example.
	
\begin{example}\label{exmp50}
	Let ${\mathcal F}=\{ \langle 5,7,9,11,13 \rangle, \langle 4,10,11,13 \rangle\}$ and $\Delta = \langle 4,5,7 \rangle$. We are going to compute the tree ${\mathrm G}({\mathcal R}({\mathcal F},\Delta))$.
	
	First of all, to compute the minimal ${\mathcal R}({\mathcal F},\Delta)$-system of generators of $\langle 4,5,7 \rangle$, we apply Proposition~\ref{prop48} with $A=\{4,5,7\}$. Since $x_{\langle 5,7,9,11,13 \rangle}=4$ and $x_{\langle 4,10,11,13 \rangle}=5$, then $\{4,5\}$ is such a minimal ${\mathcal R}({\mathcal F},\Delta)$-system. Now, because ${\mathrm F}_{\Delta}(\langle 4,5,7 \rangle)=-1$, and by applying Theorem~\ref{thm28}, we get that $\langle 4,5,7 \rangle$ has two children, $\langle 4,5,7 \rangle \setminus \{4\}=\langle 5,7,8,9,11 \rangle$ (with ${\mathrm F}_{\Delta}(\langle 5,7,8,9,11 \rangle)=4$) and $\langle 4,5,7 \rangle \setminus \{5\}=\langle 4,7,9,10 \rangle$ (with ${\mathrm F}_{\Delta}(\langle 4,7,9,10 \rangle)=5$).
	
	Now, if we take $A=\{5,7,8,9,11\}$ in Proposition~\ref{prop48}, then we have that $x_{\langle 5,7,9,11,13 \rangle}=8$ and $x_{\langle 4,10,11,13 \rangle}=5$. Thus, we conclude that $\{5,8\}$ is the minimal ${\mathcal R}({\mathcal F},\Delta)$-system of $\langle 5,7,8,9,11 \rangle$. Moreover, since ${\mathrm F}_{\Delta}(\langle 5,7,8,9,11 \rangle)=4$, then Theorem~\ref{thm28} asserts that $\langle 5,7,8,9,11 \rangle \setminus \{5\}=\langle 7,8,9,10,11,12,13 \rangle$ (with ${\mathrm F}_{\Delta}(\langle 7,8,9,10,11,12,13 \rangle)=5$) and $\langle 5,7,8,9,11 \rangle \setminus \{8\}=\langle 5,7,9,11,13 \rangle$ (with ${\mathrm F}_{\Delta}(\langle 5,7,9,11,13 \rangle)=8$) are the two children of $\langle 5,7,8,9,11 \rangle$.
	
	With $A=\{4,7,9,10\}$, we get that $\{4,7\}$ is the minimal ${\mathcal R}({\mathcal F},\Delta)$-system of $\langle 4,7,9,10 \rangle$. By recalling that ${\mathrm F}_{\Delta}(\langle 4,7,9,10 \rangle)=5$, we conclude that $\langle 4,7,9,10 \rangle$ has only one child, that is $\langle 4,7,9,10 \rangle \setminus \{7\}=\langle 4,9,10,11 \rangle$ (with ${\mathrm F}_{\Delta}(\langle 4,9,10,11 \rangle)$ $=7$).
	
	By repeating the above process, we get the whole tree ${\mathrm G}({\mathcal R}({\mathcal F},\Delta))$.
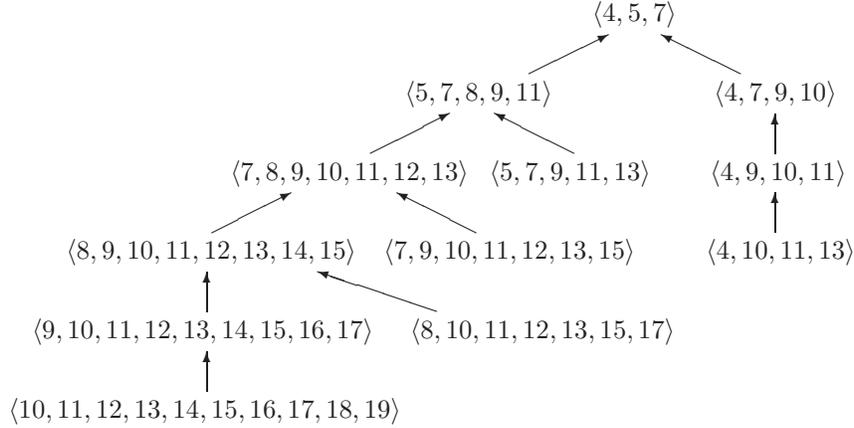
	\begin{center}
		\begin{picture}(335,155)
		\put(229,150){$\langle 4,5,7 \rangle$}
		\put(205,130){\vector(2,1){30}} \put(285,130){\vector(-2,1){30}}
		\put(158,120){$\langle 5,7,8,9,11 \rangle$} \put(275,120){$\langle 4,7,9,10 \rangle$}
		\put(145,100){\vector(2,1){30}} \put(222,100){\vector(-2,1){30}} \put(298,100){\vector(0,1){15}}
		\put(92,90){$\langle 7,8,9,10,11,12,13 \rangle$} \put(190,90){$\langle 5,7,9,11,13 \rangle$} \put(273.5,90){$\langle 4,9,10,11 \rangle$}
		\put(85,70){\vector(2,1){30}} \put(185,70){\vector(-2,1){30}} \put(298,70){\vector(0,1){15}}
		\put(30,60){$\langle 8,9,10,11,12,13,14,15 \rangle$} \put(150,60){$\langle 7,9,10,11,12,13,15 \rangle$} \put(272,60){$\langle 4,10,11,13 \rangle$}
		\put(83,40){\vector(0,1){15}} \put(170,40){\vector(-3,1){45}}
		\put(17,30){$\langle 9,10,11,12,13,14,15,16,17 \rangle$} \put(160,30){$\langle 8,10,11,12,13,15,17 \rangle$}
		\put(83,10){\vector(0,1){15}}
		\put(8,0){$\langle 10,11,12,13,14,15,16,17,18,19 \rangle$}
		\end{picture}
	\end{center}
	
	Now, we are going to represent the vertices of ${\mathrm G}({\mathcal R}({\mathcal F},\Delta))$ using their minimal ${\mathcal R}({\mathcal F},\Delta)$-systems of generators. Moreover, we add to each vertex the corresponding Frobenius number restricted to $\Delta$. Thus, we clarify all the steps to build the tree ${\mathrm G}({\mathcal R}({\mathcal F},\Delta))$.
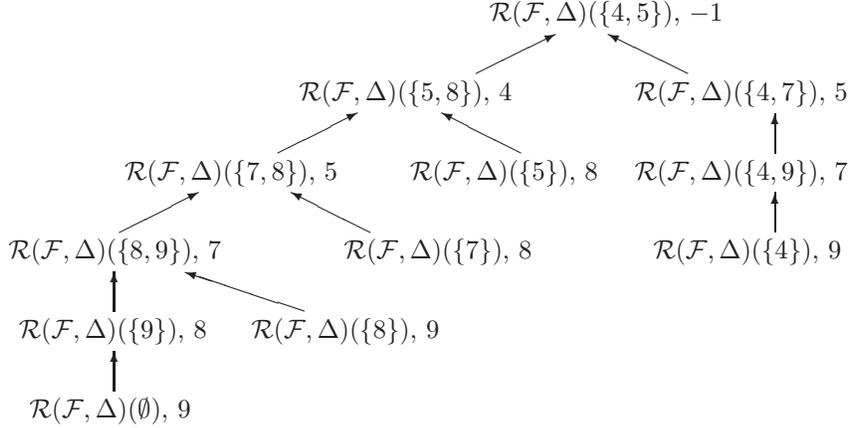
	\begin{center}
		\begin{picture}(335,155)
		\put(190,150){${\mathcal R}({\mathcal F},\Delta)(\{4,5\}),\,-1$}
		\put(185,130){\vector(2,1){30}} \put(265,130){\vector(-2,1){30}}
		\put(118,120){${\mathcal R}({\mathcal F},\Delta)(\{5,8\}),\,4$} \put(245,120){${\mathcal R}({\mathcal F},\Delta)(\{4,7\}),\, 5$}
		\put(110,100){\vector(2,1){30}} \put(202,100){\vector(-2,1){30}} \put(298,100){\vector(0,1){15}}
		\put(52,90){${\mathcal R}({\mathcal F},\Delta)(\{7,8\}),\,5$} \put(160,90){${\mathcal R}({\mathcal F},\Delta)(\{5\}),\,8$} \put(245,90){${\mathcal R}({\mathcal F},\Delta)(\{4,9\}),\,7$}
		\put(50,70){\vector(2,1){30}} \put(145,70){\vector(-2,1){30}} \put(298,70){\vector(0,1){15}}
		\put(8,60){${\mathcal R}({\mathcal F},\Delta)(\{8,9\}),\,7$} \put(135,60){${\mathcal R}({\mathcal F},\Delta)(\{7\}),\,8$} \put(252,60){${\mathcal R}({\mathcal F},\Delta)(\{4\}),\,9$}
		\put(48,40){\vector(0,1){15}} \put(120,40){\vector(-3,1){45}}
		\put(12,30){${\mathcal R}({\mathcal F},\Delta)(\{9\}),\,8$} \put(100,30){${\mathcal R}({\mathcal F},\Delta)(\{8\}),\,9$}
		\put(48,10){\vector(0,1){15}}
		\put(16,0){${\mathcal R}({\mathcal F},\Delta)(\emptyset),\,9$}
		\end{picture}
	\end{center}
	
\end{example}

\section{Conclusion}

We have been able to give a structure to certain families of numerical semigroups which are not (Frobenius) varieties or (Frobenius) pseudo-varieties. For that we have generalized the concept of Frobenius number to the concept of restricted Frobenius number and, then, we have defined the $R$-varieties (or (Frobenius) restricted variety). 

After studying relations among varieties, pseudo-varieties, and $R$-varieties, we have introduced the concepts of $R$-monoid and minimal $R$-system of generators of a $R$-monoid, which lead to associate a tree with each $R$-variety and, in consequence, obtain recurrently all the elements of an $R$-variety.

Finally, although in general it is not possible to define the smallest $R$-variety that contains a given family ${\mathcal F}$ of numerical semigroups, we have been able to give an alternative when we fix in advance the maximum of the smallest $R$-variety.


\end{document}